\documentclass[11pt]{amsart}

\usepackage{amsmath,amsfonts,amsthm,mathrsfs,amssymb,bbm,amscd}
\usepackage[all]{xy}

\newtheorem{theorem}{Theorem}
\newtheorem{lemma}[theorem]{Lemma}

\newtheorem{definition}[theorem]{Definition}

\theoremstyle{definition}

\numberwithin{equation}{section}

\renewcommand{\t}[1]{\tilde{#1}}
\newcommand{\mbb}[1]{\mathbbm{#1}}

\renewcommand{\sc}[1]{\mathscr{#1}}

\newcommand{\xr}[1]{\xrightarrow{#1}}

\title{Smash nilpotence on Uniruled 3-folds}
\begin{document}

\author[R. Sebastian]{Ronnie Sebastian}

\address{Indian Institute of Science Education and Research (IISER),
Dr. Homi Bhabha Road, Pashan, Pune 411008, India }

\email{ronnie.sebastian@gmail.com}



\begin{abstract}
Voevodsky has conjectured that numerical 
and smash equivalence coincide on a smooth 
projective variety. We prove this conjecture 
holds for uniruled 3-folds and for one 
dimensional cycles on products of Kummer surfaces.
\end{abstract}
\maketitle
\section{Introduction}
Throughout this article we work over an algebraically 
closed field $k$ and with algebraic cycles with 
rational coefficients.

Let $X$ be a smooth and projective variety over $k$. 
In \cite{voe}, Voevodsky defines a cycle $\alpha$ 
to be smash nilpotent if the cycle 
$\alpha^n:=\alpha\times\alpha\ldots\times\alpha$ 
on the variety $X^n:=X\times X\ldots\times X$ is 
rationally equivalent to 0. It is trivial to see 
that a smash nilpotent cycle is numerically trivial,
Voevodsky conjectured that the converse also holds. 
Voevodsky, \cite{voe}, and Voisin, \cite{voisin},
prove that a cycle which is algebraically trivial
is smash nilpotent. 

Kimura, \cite[Proposition 6.1]{kimura}, proved that 
a morphism between finite dimensional motives of 
different parity is 
smash nilpotent. Thus, if an algebraic 
cycle can be viewed as a morphism between motives 
of different parities, then it is smash nilpotent. 
In \cite{ks}, the authors use this fact to prove 
that skew cycles on an abelian variety are smash 
nilpotent. A cycle $\beta$ is called skew if it satisfies 
$[-1]^*\beta=-\beta$. In \cite{ks} such  
cycles are expressed as morphisms between motives of 
different parity, using the fact that the motive of 
an abelian variety has a Chow-Kunneth decomposition, 
\[h(A)=\bigoplus_{i=0}^{2\,{\rm dim}\,A}h^i(A)\]
and the motives $h^i(A)$ for $i$ odd are oddly 
finite dimensional. 

In \cite{sebastian1} it is proved that for one 
dimensional cycles on a variety dominated 
by a product of curves, smash equivalence and numerical
equivalence coincide. The same result can be deduced 
from \cite{marini} and \cite{herbaut}, 
where it is shown that for a smooth projective curve $C$, 
for any adequate equivalence relation, $[C]_{i}=0$ 
implies that $[C]_{i+1}=0$, for $i\geq2$. Here $[C]_i$ denotes 
the Beauville component of the curve $C$ in 
its Jacobian satisfying $[n]_*[C]_i=n^i[C]_i$.
If we combine this with \cite{ks}, where it 
is shown that $[C]_3=0$ modulo smash equivalence, 
then one can deduce the results in
\cite{sebastian1}. 

If we take the Chow ring of an abelian variety modulo 
algebraic equivalence and go modulo the subring 
generated by the cycles in the preceding paragraphs
under the Pontryagin product, intersection product 
and Fourier transform, then there are no nontrivial 
examples of higher dimensional cycles (dim $>$ 1) for which 
Voevodsky's conjecture holds.

The purpose of this article is to write down some 
more examples for which this conjecture holds. The 
main theorems in this article are

\begin{theorem}\label{thm1}
Let $X$ be uniruled 3-fold. Then numerical and smash 
equivalence coincide for cycles on $X$. 
\end{theorem}

\begin{theorem}\label{thm2}
Let $K_i$, $i=1,2,\ldots,N$ be Kummer surfaces. 
Then numerical and smash equivalence coincide for 
one dimensional cycles on 
$X:=K_1\times K_2\times \cdots \times K_N$.
\end{theorem}

The proof of the above theorems use Lemma \ref{l1}, 
which implies the following. If numerical and 
smash equivalence coincide on a smooth and 
projective variety $Y$, then they coincide 
on $\t{Y}$, which is obtained by blowing 
up $Y$ along a smooth subvariety of 
dimension $\leq 2$.

{\bf Acknowledgements}. We thank 
Najmuddin Fakhruddin for useful discussions.

\section{Smash equivalence and blow ups}
Let 
$Y$ be a smooth variety and $i:X\hookrightarrow Y$ 
be a smooth and closed subvariety. Let $f:\t{Y}\to Y$ 
denote the blow-up of $Y$ along $X$. 

\begin{lemma}\label{l1}
If numerical and smash equivalence coincide for elements 
in $CH_i(X)$ for $i\leq r$ and $CH_r(Y)$, then they 
coincide for elements in $CH_r(\t{Y})$.
\end{lemma}
\begin{proof}
Consider the 
Cartesian square
\[\xymatrix{\t{X}\ar[r]\ar[d]_g & \t{Y}\ar[d]^f\\
		X\ar[r]^{i} & Y}\]
Then \cite[Proposition 6.7]{fulton} says that there 
is an exact sequence 
\[0\to CH_r(X)\to CH_r(\t{X})\oplus CH_r(Y)\to CH_r(\t{Y})\to 0\]
Since $X$ is a smooth subvariety of $Y$, we 
have that $\t{X}\xr{g}X$ is the projective bundle 
associated to the locally free sheaf 
$\sc{I}_{X}/\sc{I}_X^2$ on $X$. Thus, every element  
$\beta\in CH_r(\t{X})$ may be expressed as the sum 
\[\beta=\sum_{i=0}^{d-1}c_1(\sc{O}(1))^i\cap g^*g_*(\beta\cap c_1(\sc{O}(1))^{d-1-i}),\]
where $\sc{O}(1)$ is the tautological bundle on $\t{X}$.
If we assume that numerically trivial elements 
in $CH_i(X)$ are smash nilpotent for $i\leq r$, 
then the above formula shows that numerically 
trivial elements in $CH_r(\t{X})$ are smash nilpotent. 
If numerical and smash equivalence coincide
for elements in $CH_r(Y)$, then the above exact sequence 
would show that these coincide for elements in 
$\t{Y}$ as well. 
\end{proof}

The following is a standard result which we include 
for the benefit of the reader.
\begin{lemma}\label{l2}
Let $X$ be a smooth projective variety and let $h:Y\to X$ 
be a dominant morphism. If numerical and smash 
equivalence coincide for cycles on $Y$, then they 
coincide for cycles on $X$.
\end{lemma}
\begin{proof}
Let $l\in CH^1(Y)$ be a relatively ample line bundle. 
The relative dimension of $h$ is $r:=\rm{dim}(Y)-\rm{dim}(X)$ 
and define $d$ by $h_*(l^r)=:d[Y]$. 
Then by the projection formula, we have 
$\forall \alpha\in CH^*(X)$
\[h_*(l^r\cdot h^*\alpha)=d\alpha\] 
If $\alpha$ is a numerically trivial cycle on $X$, then 
$l^{r}\cdot h^*\alpha$ is a numerically trivial cycle on 
$Y$ and so is smash nilpotent. The above equation shows 
that $\alpha$ is smash nilpotent.
\end{proof}

\section{Examples}

\subsection{Uniruled 3-folds}
\begin{definition}By a uniruled 
3-fold we mean a smooth projective variety $X$ for 
which there is a dominant rational map 
$\varphi:S\times \mbb{P}^1\dashrightarrow X$ for 
some smooth projective surface $S$. 
\end{definition}

\begin{proof}[{\bf Proof of Theorem \ref{thm1}}]Since $X$ is projective 
and $Y:=S\times \mbb{P}^1$ is normal, $\varphi$ can be 
defined on an open set $U$ whose compliment 
has codimension $\geq 2$. Let $X\hookrightarrow \mbb{P}^n$ 
be a closed immersion, composing this with $\varphi$ we 
get a morphism $g:U\to \mbb{P}^n$. Let $L$ denote the 
pullback of $\sc{O}(1)$ along $g$. Since 
$Y\setminus U$ has codimension $\geq2$, there is a 
unique line bundle on $Y$ which restricts to $L$, 
we denote this also by $L$. As $Y$ is smooth and 
codimension $Y\setminus U$ is $\geq2$, the restriction 
map $H^0(Y,L)\to H^0(U,L)$ is an isomorphism, see, 
for example \cite[Chapter 3, Ex 3.5]{hart}. Let 
$V\subset H^0(Y,L)$ be the subspace of global 
sections $g^*H^0(\mbb{P}^n,\sc{O}(1))$. Let 
$J\subset L$ be the subsheaf generated by $V$, 
then $I=J\otimes L^{-1}$ is an ideal sheaf such that 
$Y\setminus \rm{Supp}(I)=U$  
and $V$ is contained in the image of the map 
$H^0(Y,I\otimes L)\to H^0(Y,L)$ and it generates 
$I\otimes L$.

We want to apply the principalization theorem to the 
ideal sheaf $I$. In characteristic 0, see
\cite[Theorem 3.21]{kollar}, and in 
positive characteristic, see \cite[Theorem 1.3]{cutkosky}.
We get a 
morphism $f:Y'\to Y$ which is obtained as a composite 
of smooth blow-ups, such that $f^*I$ is a locally 
principal ideal sheaf and $f$ is an isomorphism on 
$f^{-1}(U)$. The subspace 
$f^*V\subset H^0(Y',f^*I\otimes f^*L)$ defines a map 
$Y'\to \mbb{P}^n$ which extends $g$. Thus, we get a 
dominant morphism $Y'\to X$. As $S$ is a surface, 
numerical and smash equivalence coincide for cycles 
on $S$ and so for cycles on $Y$. Since $Y'$ is obtained 
from $Y$ by blowing up at smooth centers and 
$\rm{dim}(Y)=3$, numerical and smash 
equivalence coincide for $Y'$ using Lemma \ref{l1}. 
Finally, use Lemma \ref{l2} to get the same result 
for $X$.
\end{proof}

\subsection{Kummer surfaces}
Let $Y$ be an abelian surface and let $X$ 
be the set of 2 torsion points. These are exactly 
the fixed points for 
the involution $x\mapsto x^{-1}$ on $Y$. This 
involution lifts to an involution of $\t{Y}$ 
which we denote $\t{i}$ and the quotient $\t{Y}/\t{i}$ 
is the Kummer surface associated to $Y$. We 
denote this surface by $K$ and by $\pi$ the 
quotient map $\t{Y}\to K$. 

Let $Y_i$, $i=1,2$ be abelian surfaces and 
let $K_i$, $i=1,2$ be the associated Kummer 
surfaces. Let $X_i$ be the set of 2 torsion 
points in $Y_i$. Similarly, we have the varieties 
$\t{Y}_i$ and there is a dominant projective 
map $\t{Y}_1\times \t{Y}_2\to K_1\times K_2$.

The map $\t{Y}_1\times \t{Y}_2 \to Y_1\times Y_2$ 
may be factored as the composite of two blow ups
\[\t{Y}_1\times \t{Y}_2 \to \t{Y}_1\times Y_2 \to Y_1\times Y_2\]
the first along the surface $X_1\times Y_2$ and 
the second along the surface $\t{Y}_1\times X_2$.
Applying Lemma \ref{l1} to both these blow ups, we get that 
numerical and smash equivalence coincide for 
one dimensional cycles on $\t{Y}_1\times \t{Y}_2$ 
and so using Lemma \ref{l2} they coincide on $K_1\times K_2$. 

\begin{proof}[{\bf Proof of Theorem \ref{thm2}}]
We recall a result from \cite{sebastian1} which we 
need. Let $N\geq 3$ be an integer and let $C$ be a 
smooth projective curve with a base point $c_0$. 
Let $\Delta_C$ denote the diagonal embedding 
$C\hookrightarrow C^N$. Let $p_{ij}:C^N\to C^N$ 
denote the map which leaves the $i$th and $j$th 
coordinates intact and the other coordinates are 
changed to $c_0$, for example, 
$p_{12}(x_1,x_2,\ldots,x_N)=(x_1,x_2,c_0,c_0,\ldots,c_0)$.
Then there are rational numbers $q_{ij}$ such that
\begin{equation}\label{e1}\Delta_C\sim_{sm}\sum_{i\neq j}q_{ij}p_{ij*}(\Delta_C).\end{equation}

Let $X:=K_1\times K_2\ldots\times K_N$ be a product 
of Kummer surfaces.
Fix base points $e_i\in K_i$, and define (we abuse notation here)
$p_{ij}:X\to X$ 
in the same way as above, using these base points. We 
remark that if we work modulo algebraic equivalence, 
for any cycle $\alpha\in CH_*(X)$, the cycle 
$p_{ij*}(\alpha)$ is independent of the choice 
of these base points. Hence, the same is true modulo 
smash equivalence.

Let $D\hookrightarrow X$ be a reduced and irreducible 
one dimensional subvariety. Let $C\to D$ be its normalization 
and denote the composite map by $f:C\to X$. If we let 
$p_i$ denote the projection from $X$ to $K_i$ and let 
\[\pi:=(p_1\circ f)\times (p_2\circ f)\ldots(p_n\circ f),\]
then we get $f_*([C])=\pi_*(\Delta_C)$. Using equation 
\eqref{e1}, we get that modulo smash equivalence
\begin{align*}
	[D]=f_*([C])=&\pi_*(\Delta_C)=\sum_{i\neq j}q_{ij}\pi_*p_{ij*}(\Delta_C)=\sum_{i\neq j}q_{ij}p_{ij*}\pi_*(\Delta_C)\\
		=&\sum_{i\neq j}q_{ij}p_{ij*}([D])
\end{align*}
In particular, we get for any one dimensional cycle $\alpha$,
\[\alpha=\sum_{i\neq j}q_{ij}p_{ij*}(\alpha)\]
As we have seen above, on a product of two Kummer surfaces, numerical and 
smash equivalence coincide for one dimensional cycles. If $\alpha$ is numerically trivial, 
each $p_{ij*}(\alpha)$ is numerically trivial and so smash 
nilpotent. Thus, $\alpha$ is smash nilpotent.
\end{proof}

\end{document}